\documentclass[12pt,psamsfonts]{amsart}
\usepackage{amsmath}
\usepackage{amsthm}
\usepackage{amssymb}
\usepackage{amscd}
\usepackage{amsfonts}
\usepackage{amsbsy}
\usepackage{graphicx}
\usepackage{subfigure,graphicx}
\usepackage{overpic}
\usepackage{color}
\usepackage{caption}%
\usepackage[font=small,labelfont=bf]{caption}
\usepackage[normalem]{ulem}

\usepackage{float}

\DeclareCaptionFont{tiny}{\tiny}

\usepackage[dvips]{psfrag}
\usepackage{array}
\usepackage{epsfig}
\usepackage{url}
\usepackage{overpic}
\usepackage{tikz-cd}
\usepackage{enumitem}
\usepackage{hyperref}
\usepackage{soul}

\newtheorem {theorem} {Theorem}
\newtheorem {proposition} [theorem]{Proposition}

\newtheorem {remark} [theorem]{Remark}

\usepackage{tabularx}
\newcommand{\R}{\ensuremath{\mathbb{R}}}

\newcommand{\CC}{\mathcal{C}}

\newcommand{\CF}{\ensuremath{\mathcal{F}}}

\newcommand{\CO}{\ensuremath{\mathcal{O}}}
\newcommand{\CZ}{\ensuremath{\mathcal{Z}}}

\newcommand{\ov}{\overline}

\newcommand{\T}{\theta}
\newcommand{\f}{\varphi}
\newcommand{\al}{\alpha}

\newcommand{\cl}{\mbox{\normalfont{Cl}}}
\newcommand{\s}{\ensuremath{\mathbb{S}}}

\newcommand{\U}{\ensuremath{\mathcal{U}}}

\newcommand{\bx}{{\bf x}}
\newcommand{\by}{{\bf y}}
\newcommand{\bz}{{\bf z}}
\newcommand{\bg}{{\bf g}}
\newcommand{\bF}{{\bf F}}

\def\p{\partial}
\def\e{\varepsilon}

\usepackage{mathpazo}

\title[Zero-Hopf bifurcation in Van der Pol-Duffing equation]
{Zero-Hopf bifurcation in the general Van der Pol-Duffing equation}

\author[M.R. C\^{a}ndido$^{1}$ and C. Valls$^{2}$]
{Murilo R. C\^{a}ndido$^{1}$ \and Claudia Valls$^{2}$}

\address{$^1$ Departamento de Matem\'{a}tica, Universidade
Estadual de Campinas, Rua S\'{e}rgio Buarque de Holanda, 651, Cidade Universit\'{a}ria Zeferino Vaz, 13083--859, Campinas, SP,
Brazil} \email{candidomr@ime.unicamp.br}
\address{$^{2}$ Instituto Superior T\'ecnico, Universidade de Lisboa, 1049-001 Lisboa, Portugal}
\email{cvalls@math.tecnico.ulisboa.pt}

\subjclass[2010]{}


\begin{document}

\maketitle

\begin{abstract}
We study analytically the coexistence of multiple periodic solutions and invariant tori in the general Van der Pol-Duffing oscillator equations. 
We use  several results related to the averaging method in order to analytically obtain our results. We also provide numerical examples for all the analytical results that we provide.
\end{abstract}

\section{Introduction and statements of the main results}\label{sec1}
In three-dimensional autonomous differential systems a \textit{zero-Hopf equilibrium} is a equilibrium point which has a zero eigenvalue and a pair of purely imaginary eigenvalues. Generically, a \textit{zero-Hopf bifurcation} is a two-parameter unfolding of a three-dimensional autonomous differential equation with a zero-Hopf equilibrium. The scientific literature about zero-Hopf bifurcation is rich and varied, containing works  of several authors for instance Guckenheimer and Holmes \cite{guck81,guck84}, Scheurle and Marsden \cite{sche84}, Han \cite{maoan98} and many others.

However, there is no general theory that is able to describe completely the variety of invariant sets that may bifurcate from zero-Hopf equilibria. Consequently, most of the systems exhibiting  a zero-Hopf bifurcation require to be studied directly. The complexity of the invariant sets that can emerge from zero-Hopf equilibria make this kind of bifurcation a classical indication of local birth of chaos in differential systems. Although, the proof of generic unfolding of Hopf-zero singularities displaying strange attractors is fairly recent (see \cite{baldo20}).

In this paper we investigate the existence of periodic solutions and invariant tori emerging from the origin of the three-dimensional differential system
\begin{equation}\label{gvd}
\begin{aligned}
\dot{x}=&-\nu  \left(x^3-\mu  x-y\right),\\
\dot{y}=&-h z+k x-\alpha  y,\\
\dot{z}=&\beta  y,
\end{aligned}
\end{equation}
where $\alpha$, $h$, $\beta$, $k$, $\nu$ and $\mu$ are real parameters. This system is a generalization of the Van der Pol-Duffing oscillator  proposed by Zhao et. al. in \cite{zly}.

The dynamical behaviour of system \eqref{gvd} was investigated by Matouk and Agiza in \cite{ma} and  Zhao et. al. in \cite{zly}. They verified the existence of periodic solutions, gave the stability conditions for the equilibrium points and  investigated  the presence of chaos and hidden attractors. These papers used a combination of analytical and numerical  methods for obtaining their results.

Our approach is focused in the analytic detection of multiple invariant sets bifurcating from the origin of system \eqref{gvd}. In our main result, three periodic orbits and two invariant tori are detected bifurcating simultaneously from the origin of system \eqref{gvd}. As far as we know, this is the first time that this kind of phenomena has  been reported for system \eqref{gvd}.  The technique used in this paper is a natural combination of the averaging theory and Lyapunov coefficient method  that was recently developed  in \cite{ma}.

Take $\omega>0$, it can be easily  verified  that there are five parameter families  such that the origin of system \eqref{gvd} is a zero-Hopf equilibrium point. Namely, the parameter families are

\begin{itemize}
\item[$i$)] $h=0$, $k=-\dfrac{\mu ^2 \nu ^2+\omega ^2}{\nu}$ and $\al= \mu  \nu$;
\item[$ii)$] $\beta= 0$, $k= -\dfrac{\mu ^2 \nu ^2+\omega ^2}{\nu }$, and $\al=\mu   \nu$;
\item[$iii)$] $\alpha = 0$, $\beta = \dfrac{k \nu +\omega ^2}{h}$ and $\mu = 0$;
\item[$iv)$]$\al= 0$, $\beta=\frac{\omega ^2}{h}$ and $\nu= 0$;
\item[$v)$] $\al=0$, $k=-\dfrac{\omega ^2}{\nu }$, $ h=0$  and $\mu=0$.
\end{itemize}

In order to study system \eqref{gvd} with coefficients that are $\e$-close to the families presented above we are going to assume hereafter that the coefficients of system \eqref{gvd} depend smoothly on $\e$, accordingly
\begin{equation}\label{coefs}
\begin{aligned}
h(\e)&= h_0+\sum_{i=1}^k\e^i h_i+\CO(\e^{k+1}), &
k(\e)&=k_0+\sum_{i=1}^k\e^i k_i+\CO(\e^{k+1}), \\
\alpha(\e)&=\al_0+\sum_{i=1}^k\e^i \alpha_i+\CO(\e^{k+1}), &
\beta(\e)&=\beta_0+\sum_{i=1}^k\e^i \beta_i+\CO(\e^{k+1}),\\
\mu(\e)&=\mu_0+\sum_{i=1}^k\e^i \mu_i+\CO(\e^{k+1}), &
\nu(\e)&=\nu_0+\sum_{i=1}^k\e^i \nu_i+\CO(\e^{k+1}).
\end{aligned}
\end{equation}
Our first result will be on the coexistence of multiple tori and periodic solutions in system \eqref{gvd}. In this case, three periodic solutions will coexists if  the quantities
\begin{equation}\label{param}
\begin{aligned}
&\delta_a=k_0 \left(2 k_0 \mu_1 \nu_0^2 -\al_1 \omega ^2\right) , & &\delta_b=k_0 \left(2 \al_1 \omega ^2+k_0 \mu_1 \nu_0^2\right),\\
&\delta_c=k_0 \left(\al_1 \omega ^2+k_0 \mu_1 \nu_0^2\right), & &\delta_d=\al_1\left(k_0 \nu_0+\omega ^2\right),
\end{aligned}
\end{equation}
are simultaneously positive. In this case, two periodic solutions will be symmetric with respect to the central one. The stability of the central periodic solution will be determined by the sign of the following values
\begin{align*}
\lambda_1 &=\frac{\left(k_0 \nu_0 +\omega ^2\right) \left(2 \al_1 \omega ^2+5 \al_1 k_0 \nu_0 +k \mu_1 \nu_0 ^2\right)}{k_0 \nu_0  \omega ^3} ,\\
\lambda_2  &= -\frac{8 \al_1 \omega ^2+15 \al_1 k_0 \nu_0 +2 k_0 \mu_1 \nu_0 ^2}{4 \omega ^3}.
\end{align*}
The other two symmetric periodic solutions will have their stability determined by the following coefficients
\begin{equation*}
\begin{split}
A & =
\frac{\left(k_0 \nu_0 +2 \omega ^2\right) \left(4 \al_1 \omega ^2+k \nu_0  (5 \al_1+2 \mu_1 \nu_0 )\right)}{8 k_0 \nu_0  \omega ^3}, \\
B & =\dfrac{1}{\al_1^2 k_0^2 \nu_0 ^2}\Big(64 \al_1^2 \omega ^8+k_0^4 \nu_0 ^4 \left(145 \al_1^2+276 \al_1 \mu_1 \nu_0 +36 \mu_1^2 \nu_0 ^2\right)\\
& +4 k^3 \nu_0 ^3 \omega ^2 \left(153 \al_1^2+136 \al_1 \mu_1 \nu_0 +12 \mu_1^2 \nu_0 ^2\right)\\
& +4 k^2 \nu_0 ^2 \omega ^4 \left(221 \al_1^2+84 \al_1 \mu_1 \nu_0 +4 \mu_1^2 \nu_0 ^2\right)+32 \al_1 k_0 \nu_0  \omega ^6 (15 \al_1+2 \mu_1 \nu_0 )\Big).
\end{split}
\end{equation*}
When $A$ passes through zero these periodic solutions will change their stability and a Neimark-Sacker bifurcation may occur around each symmetric periodic solution. In this case, two invariant tori will emerge. In this situation the bifurcation parameters are
\begin{equation*}
\begin{split}
\ell_{1,1}&=-\frac{ h_0^2 \nu_0  \left(5 k_0^2 \nu_0 ^2+16 k_0 \nu_0  \omega ^2+8 \omega ^4\right)}{ \left(k_0 \nu_0 +\omega ^2\right)},\\
 \widehat{\mu}_1&=-\frac{\al_1 \left(5 k_0 \nu +4 \omega ^2\right)}{2 k_0 \nu_0 ^2}.
\end{split}
\end{equation*}
Our first main result can be stated as follows.
\begin{theorem}\label{tvd1}
Consider system \eqref{gvd} with the parameters \eqref{coefs} satisfying the following relation
$$
\alpha_0 = 0, \quad \beta_0 = \dfrac{k_0 \nu_0 +\omega ^2}{h_0} \quad \mbox{and} \quad \mu_0 = 0.
$$
Then, the following statements hold:
\begin{itemize}
\item[a)] If $\delta_a>0,$ $\delta_b>0,$ $\delta_c>0$ and $\delta_d>0,$  for $\e>0$ sufficiently small, system \eqref{gvd} will have three periodic solutions $\varphi_0(t,\e),$  $\varphi_+(t,\e)$ and $\varphi_-(t,\e)$ bifurcating from the origin.

\item[b)] If $\lambda_1<0$ and $\lambda_2<0$, then the periodic solution $\varphi_0(t,\e)$ will be asymptotic stable. On the other hand, if $\lambda_1>0$, or $\lambda_2>0$ the periodic solution $\varphi_0(t,\e)$ will be unstable. Analogously, we define
$$ \lambda_\pm=A\pm \dfrac{\al_1}{8\omega^3}\sqrt{B}.$$
If $\Re(\lambda_+)<0$ and $\Re(\lambda_-)<0$, then both periodic solutions $\varphi_\pm(t,\e)$ will be stable. In case that $\Re(\lambda_+)>0$, or $\Re(\lambda_-)>0$ both periodic solutions will be unstable.

\item[c)] If $\ell_{1,1}\neq 0$, then there exists a smooth curve $\widehat{\mu}(\e)$, satisfying $\widehat{\mu}(\e)=\e \widehat{\mu}_1+\CO(\e^2)$, such that each of the periodic orbits $\varphi_\pm(t,\e)$ bifurcates into an unique invariant torus whenever $\ell_{1,1}(\mu_1-\widehat{\mu}(\e))<0$. Moreover, if $\ell_{1,1}>0$ (resp. $\ell_{1,1}<0$) these tori are unstable (resp. asymptotically stable), whereas the periodic orbits $\varphi_\pm(t,\e)$ are asymptotically stable (resp. unstable).
\end{itemize}
\end{theorem}
The proof of Theorem \ref{tvd1} will be provided in Section \ref{pr} (see Proposition \ref{p1} and \ref{p5}). The next result provides information about the bifurcation of periodic solutions associated with the parametric families $(i),$ $(ii),$ $(iv)$ and $(v)$.  We introduce the following assumptions.
\begin{itemize}
\item[{\bf H$_1$}.] {\it If $ \delta_1=\omega ^2 (\mu_0 \nu_1-\al_1+\mu_1 \nu_0)-\beta_0 h_1 \mu_0 \nu_0$, then the coefficients \eqref{coefs} of system \eqref{gvd} satisfy
$$
h_0=0, \quad k_0=-\dfrac{\mu_0^2 \nu_0^2+\omega ^2}{\nu_0}, \quad \al_0=\mu_0\nu_0,
$$
where  $\delta_1 \nu_0>0$ and $\beta_1h_1\mu_0\neq 0$.}

\item[{\bf H$_2$}.] { \it If $\delta_2=\omega ^2 (\mu_0 \nu_1-\al_1+\mu_1 \nu_0)-\beta_1 h_0 \mu_0 \nu_0$, then the coefficients \eqref{coefs} of system \eqref{gvd} satisfy
$$
\beta_0=0, \quad k_0=-\dfrac{\mu_0^2\nu_0^2+\omega^2}{\nu_0}, \quad \al_0=\mu_0\nu_0.
$$
Moreover, $\al_1=0$, $\al_2=0$, $\beta_0=\dfrac{\omega^2}{h_0}$ and $\mu_0=0$ where $\delta_2 \nu_0>0$ and $\beta_1h_0\mu_0\neq 0$.}

\item[{\bf H$_3$}.] {\it If $\delta_3=\sqrt{\mu_1}  (h_1 k_0-h_0 k_1)h_0 \omega$, then the coefficients \eqref{coefs} of system \eqref{gvd} satisfy
$$
\al_0=0,\quad  \beta_0=\dfrac{\omega^2}{h_0}, \quad \nu_0=0.
$$
Moreover, $\al_1=0,$  $ \mu_0=0$, $\al_2=0$, where $\delta_3>0$ and $\left(2 k_0 \mu_1 \nu_1^2-\al_3 \omega ^2\right)\neq 0$.}

\item[{\bf H$_4$}.]{ \it If $ \delta_4=\nu_0( \mu_1 \nu_0- \al_1)$, then the coefficients \eqref{coefs} of system \eqref{gvd} satisfy
$$
\al_0=0, \quad k_0=-\dfrac{\omega^2}{\nu_0}, \quad h_0=0, \quad \mu_0=0,
$$
where $\delta_4>0$ and  $ \beta_0 h_1 \left(2 \al_1^2-3 \al_1 \mu_1 \nu_0+\mu_1^2 \nu_0^2\right)\neq 0$.}
\end{itemize}

Our second main result is stated as follows.

\begin{theorem}\label{tvd2}
 Assume that system \eqref{gvd} with coefficients \eqref{coefs} satisfies one of the hypothesis {\bf H$_1$}-{\bf H$_4$}. Then for $\e>0$ sufficiently small, system \eqref{gvd} will have a periodic solution emerging from the origin of coordinates.
\end{theorem}

The next section will be dedicated to the averaging method that will be used along the paper for obtaining our results. The proofs of Theorems \ref{tvd1} and \ref{tvd2} will be presented in Section \ref{pr}. Finally, in Section \ref{sec:ex}, for each of the results stated in Theorems \ref{tvd1} and \ref{tvd2} we present a numerical example.

\section{Averaging Theory: Periodic Solutions and Invariant Tori}\label{av}

In this work we use two aspects of the averaging theory. The first aspect of the theory is the detection of periodic solutions. In this context we will use  classical averaging theory for detecting periodic solutions. Moreover,  a new version of averaging theory combined with Lyapunov-Schmidt reduction will be used for finding periodic solutions that are undetectable by the classical method \cite{can17} (see Proposition \ref{p4}).  The second aspect of the averaging theory is the detection of invariant tori. The result we will used for this task is also a new feature of averaging theory combined with the analysis of Lyapunov constants, see \cite{can17}. 

This section is dedicated to present all results that will be used for proving Theorems \eqref{tvd1} and \eqref{tvd2}.

\subsection{Existence of Periodic Solutions}
Consider the non-autonomous differential systems written in the following {\it standard form}:
\begin{equation}\label{tm13}
\dot {\bf x}=\sum_{i=1}^k \e^i {\bf F}_i(t,{\bf
x})
+\e^{k+1} \widetilde{{\bf
F}}(t,{\bf x},\e), \quad  (t,{\bf x},\e)\in\R\times\Omega\times(-\e_0,\e_0),
\end{equation}
where $\Omega$ is an open bounded subset of $\mathbb{R}^n$ and  $\e_0$ is a small positive real number. It is assumed that $\bF_i,$ $i=1,\ldots,k,$ and $\widetilde\bF$ are sufficiently smooth functions and $T$-periodic in the variable $t.$ The periodicity of system \eqref{tm13} allow us to see it as defined in the cylinder $(t,\bx) \in \s^1\times \Omega,$ where $\s^1\equiv \R/T\mathbb{Z}.$

From the qualitative theory of ordinary differential equations, the solution $\bx(t,\bz,\e)$ of \eqref{tm13}, satisfying $\bx(0,\bz,\e)=\bz$, can be written as
\begin{equation*}\label{tri}
\bx(t,\bz,\e)=\bz+\sum_{i=1}^k\e^i \dfrac{\by_i(t,\bz)}{i!}+\CO(\e^{k+1}),
\end{equation*}
where the expressions for $\by_i$ are obtained by solving recursively the system of equations obtained from \eqref{tm13} (see \cite[Lemma 5]{can17}).
Hence, the correspondent Poincar\'{e} map $\Pi(\bz,\e)=\bx(T,\bz,\e)$ of system \eqref{tm13} can be written as
\begin{equation}\label{quad}
\Pi(\bz,\e)=\bz+\sum_{i=1}^k\e^i\bg_i(\bz) +\CO(\e^{k+1}).
\end{equation}

It is well know, again from the qualitative theory of ordinary differential equations, that a fixed point of the Poincar\'e map \eqref{quad}, that is, a point satisfying $\Pi(\bz(\e),\e)=\bz(\e)$, corresponds to a branch of isolated $T$-periodic solutions of system \eqref{tm13}. The averaging method is based on using the coefficient functions of the Poincar\'e map \eqref{quad} in order to obtain fixed points for it.

We define by
\begin{equation}\label{avf}
\bg_i(\bz)=\dfrac{\by_i(T,\bz)}{i!}
\end{equation}
the \textit{averaged functions of order $i$}  of system \eqref{tm13}. If for some $m\in\{1,2,\ldots,k\}$ we have that $\bg_0=  \cdots= \bg_{m-1}= 0$ and $\bg_m\neq0$, then a simple zero of $\bg_m(\bz)$ provides a branch of fixed points $\bz(\e)$ for the map \eqref{quad}.

 In fact, with this terminology the classical averaging method for finding periodic solutions can be summarized by the following theorem.
\begin{theorem}[\cite{DLT}]\label{CAT}
Assume that, for some $m\in\{1,\ldots,k\}$, $\bg_0=\cdots\bg_{m-1}=0$ and $\bg_m\neq0$. If there exists $\bz^*\in \Omega$
such that $\bg_m(\bz^*)=0$ and $| D\bg_m(\bz^*)|\neq 0$, then for $|\e|\neq0$ sufficiently small there exists an isolated $T$-periodic solution
$\f(t,\e)$ of system \eqref{tm13} such that $\f(0,0)=\bz^*$.
\end{theorem}

Now we present the expressions of the functions $\by_i,$ for $i=1,\dots,4$ that will be needed in this work. For the general expressions, the reader is addressed to \cite{can17}. Consider the vector $\by=(y^1,\dots, y^n)\in\R^n$. We denote $\by^m=\big(\by, \cdots, \by \big)\in \R^{m n}$. Then
\begin{equation}\label{avf}
\bg_i(\bz)=\dfrac{\by_i(T,\bz)}{i!},
\end{equation}
where
\begin{align*}
\by_1(t,\bz)=&\int_0^t \bF_1(\tau,\bz)\mathrm{d}\tau,\\
\by_2(t,\bz)=&\int_0^t
2\bF_2(\tau,\bz)+2\dfrac{\p
\bF_1}{\p
\bx}(\tau,\bz)\by_1(\tau,\bz) \mathrm{d}\tau,\\
\by_3(t,\bz)=&\int_0^t
6\bF_3(\tau,\bz)+6\dfrac{\p
\bF_2}{\p
\bx}(\tau,\bz)\by_1(\tau,\bz)\\
&+3\dfrac{\p^2 \bF_1}{\p
\bx^2}(\tau,\bz)\by_1(\tau,\bz)^2+3\dfrac{\p \bF_1}{\p
\bx}(\tau,\bz)\by_2(\tau,\bz)\mathrm{d}\tau,\\
\by_4(t,\bz)=&\int_0^t
24\bF_4(\tau,\bz)+24\dfrac{\p \bF_3}{\p
\bx}(\tau,\bz)\by_1(\tau,\bz) +12\dfrac{\p^2 \bF_2}{\p
\bx^2}(\tau,\bz)\by_1(\tau,\bz)^2+\\
&12\dfrac{\p \bF_2}{\p
\bx}(\tau,\bz)\by_2(\tau,\bz)+12\dfrac{\p^2 \bF_1}{\p \bx^2}(\tau,\bz)\by_1(\tau,\bz)\odot
\by_2(\tau,\bz)\\
&+4\dfrac{\p^3 \bF_1}{\p
\bx^3}(\tau,\bz)\by_1(\tau,\bz)^3+4\dfrac{\p \bF_1}{\p \bx}(\tau,\bz)\by_3(\tau,\bz)\mathrm{d}\tau.
\end{align*}
The previous result cannot be used when the zero of the first non vanishing averaged function $\bg_m$ is not isolated. In this context we can use the method presented in \cite{can17} that we reproduce here.

Let $\pi:\R^m\times\R^{n-m} \rightarrow\R^m$ and
$\pi^{\perp}:\R^m\times \R^{n-m} \rightarrow\R^{n-m}$  be the
projections onto the first $m$ coordinates and onto the last $n-m$
coordinates, respectively. Denote $\bz\in \Omega$ as $\bz=(u,v)\in \R^m\times\R^{n-m}$.  Let $\ell$ be the first subindex such that $\bg_\ell\not\equiv 0$. Moreover, assume that the averaged function $\bg_\ell$ vanishes on the set
\begin{equation*}\label{grp}
\CZ=\lbrace \bz_u =(u,\mathcal{B}(u)):u \in \ov V  \rbrace\subset U,
\end{equation*}
where $m<n$ are positive integers, $V$ is an open bounded subset of $\R^m,$ and  $\mathcal{B}\colon \ov V\rightarrow \R^{n-m}$ is a  $\CC^k$ function. Thus, $\CZ$ is a set of non-isolated zeros of $\bg_\ell$ and, consequently, Theorem \ref{CAT} cannot be applied. Nevertheless, in this case we follow \cite{can17}, and the Lyapunov-Schmidt reduction can be used to obtain sufficient conditions for the existence of isolated $T$-periodic solutions bifurcating from $\CZ$ as follows. First, notice that the equation $\bz=\Pi(\bz,\e)$ is equivalent to the following system of equations
\begin{equation}\label{SE}
\left\{\begin{array}{l}
0=u-\pi \Pi(u,v,\e),\vspace{0.2cm}\\
0=v-\pi^{\perp} \Pi(u,v,\e).\\
\end{array}\right.
\end{equation}
Under convenient assumptions, the {\it implicit function theorem} can be used to find a function $\ov{\mathcal{B}}(u,\e),$ satisfying $\ov{\mathcal{B}}(u,0)=\mathcal{B}(u),$ which solves the second line of system \eqref{SE}, that is,  $\ov{\mathcal{B}}(u,\e)=\pi^{\perp} \Pi(u,\ov{\mathcal{B}}(u,\e),\e)$. Then, substituting $v=\ov{\mathcal{B}}(u,\e)$ into the first line of system \eqref{SE}, we obtain a single equation to be solved, namely $\CF(u,\e)=u- \pi \Pi(u,\ov{\mathcal{B}}(u,\e),\e)=0.$ Expanding $\CF(u,\e)$ around $\e=0,$ we get the bifurcation functions $f_i$, for $i=1,\ldots,k-1,$
\begin{align*}
\CF(u,\e)=\sum_{i=1}^{k-1} \e^i f_i(u)+\CO(\e^{k+1}),
\end{align*}
which will be given in terms of the derivatives $\gamma_j(u)=(\partial^j\ov{\mathcal{B}}/\partial \e^j)(u,0),$ $j=1,\ldots,i,$ and averaged functions $\bg_j,$ $j=1,\ldots,i+1.$ Denote $f_0=0.$ Notice that, if for some $r\in\{1,\ldots,k\}$ we have $f_0=\cdots= f_{r-1}=0$ and $f_r\neq0$, then a simple zero of $f_r(u)$ provides a branch of zeros $u(\e)$ of $\CF$, that is, $\CF(u(\e),\e)=0.$ Consequently, $\bz(\e)=\big(u(\e),\ov{\mathcal{B}}(u(\e),\e)\big)$ is a branch of solutions of system \eqref{SE}, that is, fixed points for the map \eqref{quad}. Again, $\bx(t,\bz(\e),\e)$ corresponds to a branch of isolated $T$-periodic solutions of system \eqref{tm13}. In this paper, we present the second order version of Theorem \cite[Theorem $A$]{can17}, this will be the content of Theorem \ref{LSt1}.

The next result uses some bifurcation functions $\gamma_i$ and $f_i$ for $i=1,2$ that will be defined as follows. Denote
\[
D \bg_\ell({\bf z}_u)=\begin{pmatrix}
\Lambda_u & \Gamma_u \\
B_u & \Delta_u
\end{pmatrix}
\]
where $\Lambda_u=\partial_a\pi \bg_\ell(z_u)$, $\Gamma_u=\partial_b\pi
\bg_\ell(z_u)$, $B_u=\partial_a\pi^\perp \bg_\ell(z_u)$ and
$\Delta_u=\partial_b\pi^\perp \bg_\ell(z_u)$.
The  {\it bifurcation function of order $i$} $f_i,$ for $i=1,\ldots 4,$ are defined as
\begin{equation*}
\begin{array}{rl}
\gamma_1(u)=&-\Delta_u^{-1}\pi^\perp \bg_{\ell+1}(\bz_u), \vspace{0.3cm}\\
f_1(u)=&\Gamma_u\gamma_1(u)+\pi {\bf g}_{\ell+1}({\bf z}_u),\\
\end{array}
\end{equation*}
\begin{equation*}
\begin{array}{rl}
\gamma_2(u)=&-\Delta_u^{-1}\Bigg(\dfrac{\p^2\pi^\perp
\bg_\ell}{\partial b^2}(\bz_u)\gamma_1(u)^2
+2\dfrac{\partial\pi^\perp \bg_{\ell+1}}{\partial
b}(\bz_u)\gamma_1(u)+2\pi^\perp \bg_{\ell+2}(u)
\Bigg),\vspace{0.3cm}\\
f_2(u)=&\dfrac{1}{2}\Gamma_u\gamma_2(u)+\dfrac{1}{2}\dfrac{\p^2\pi
\bg_\ell}{\p b^2}(\bz_u)\gamma_1(u)^2+\dfrac{\partial\pi
\bg_{\ell+1}}{\partial
b}(z_u)\gamma_1(u)+\pi \bg_{\ell+2}(\bz_u). \vspace{0.3cm}\\
\end{array}
\end{equation*}
Then we can define the function
\begin{equation}\label{cF}
\CF^2(\al,\e)= \e f_1(\al)+ \e^2 f_2(\al),
\end{equation}
and state the following theorem.

\begin{theorem}\label{LSt1}
Let $\Delta_{\al}$ denote the lower right corner $(n-m)\times (n-m)$
matrix of the Jacobian matrix $D\,g_i(z_{\al})$. In additional to
hypothesis (H$_a$) we assume that
\begin{itemize}
\item[$(i)$] for each $\al \in\cl(V) $, $\det(\Delta_{\al})\neq0$;

\item[$(ii)$] $f_1$ is not identically zero;

\item[$(iii)$] there exists a small parameter $\e_0>0$  such that for each
$\e\in[-\e_0,\e_0]$ there exists $a_{\e}\in V$ satisfying
$\CF^2(a_{\e},\e )=0$;

\item[$(iv)$] there exist a constant $P_0>0$ and a positive integer  $l\leq
2$ such that
\begin{equation*}
\left|\p_\al\CF^2(a_\e,\e)\cdot \al\right|\geq P_0|\e|^l |\al|, \quad
\text{for} \quad \al\in V.
\end{equation*}
\end{itemize}
Then, for $|\e|\neq0$ sufficiently small, there exists $z(\e)$ such
that $g(z(\e),\e)=0$ with
$|\pi^{\perp}z(\e)-\pi^{\perp}z_{a_{\e}}|=\CO(\e)$ and
$|\pi\,z(\e)-\pi\,z_{a_{\e}}|=\CO(\e^{k+1-l}).$
\end{theorem}

\begin{remark}\label{rem1} Let $\bz(\e)$ be a fixed point of the Poincar\'e map, using the higher order averaged functions in order to obtain an expression\begin{equation}\label{zk}
\bz(\e)=\bz_0+\e \bz_{1}+\e^2\bz_{2}+\CO(\e^3).
\end{equation}
for the fixed point. Consider also the Jacobian matrix
$$
M(\e)=I_n+\sum_{i=1}^k\e^iD\bg_i(\bz(\e)) +\CO(\e^{k+1}).
$$
If all eigenvalues of $M(\e)$ are simple, we can use the Taylor expansion of the eigenvalues of $M(\e)$ in order to determine the stability of the periodic solution $\bx(T,\bz(\e),\e)$. This method was used for finding the stable periodic solutions presented in Section \ref{sec:ex}. Note that this can be done only if all eigenvalues are simple.
\end{remark}

\subsection{Bifurcation of an Invariant Torus}\label{sec:it}
In this section we consider a parameter dependence on \eqref{tm13} and take $n=2$ obtaining the following two-parameter family of non-autonomous differential systems:
\begin{equation}\label{tm1}
\dot \bx= \e \bF_1(t,\bx,\mu)+\e^2 \widetilde{\bf
F}(t,\bx,\e,\mu),
\end{equation}
where $\Omega$ is an open bounded subset of $\mathbb{R}^2,$ $J$ is a open interval, and $\e_0$ is a small positive real number.  Notice that system \eqref{tm1} is written in the standard form \eqref{tm13} of the averaging theory with an additional  distinguished parameter $\mu$. It has been provided generic conditions on the averaged functions \eqref{avf} guaranteeing the existence of a codimension-one bifurcation curve $\mu(\varepsilon)$ in the parameter space $(\mu,\varepsilon)$ characterized by the birth of an invariant torus from a periodic solution of \eqref{tm1} in $\s^1\times \Omega$. In this section, we first introduce the main result obtained in \cite{ITCanNov2018}, and then we apply it to conclude the proof of Theorem \ref{tvd1}.

The strategy followed by \cite{ITCanNov2018} consisted in looking for conditions that ensure a {\it Neimark-Sacker Bifurcation} in the Poincar\'{e} map $\Pi(\bz, \e,\mu)=\bx(T,\bz,\e,\mu),$ $\bz\in\Sigma=\{t=0\},$ of system \eqref{tm1}. Here, $\bx(t,\bz,\e,\mu)$ denotes the solution of \eqref{tm1} satisfying $\bx(0,\bz,\e,\mu)=\bz.$ In discrete dynamical system theory, this bifurcation is characterized by the birth of an invariant closed curve from a fixed point, as the fixed point changes stability. As it well known,
an invariant torus corresponds to an invariant closed curve $\Gamma\subset\Sigma$ of $\Pi(\bz,\e,\mu)$, that is, $\Pi(\Gamma,\e,\mu)=\Gamma.$

Recall that the first-order averaged function can be written as
$$
\bg_1(\bx,\mu)=\big(\bg_1^1(\bx,\mu),\bg_1^2(\bx,\mu)\big)=\int_0^T\bF_1(t,\bx,\mu)dt.
$$

Consider the following hypotheses:
 \begin{itemize}
\item[{\bf A1.}]  {\it There exists a continuous curve $\mu\in J \mapsto\bx_\mu \in \Omega,$ defined in an interval $J\ni \mu_0,$ such that $\bg_1(\bx_\mu,\mu)=0$ for every $\mu\in J\subset\R,$  the pair of complex conjugated eigenvalues  $\al(\mu)\pm i \beta(\mu)$ of $D_\bx\bg_1(\bx_\mu,\mu)$  satisfies $\al(\mu_0)=0$ and $\beta(\mu_0)=\omega_0,$ and  $D_\bx\bg_1(\bx_{\mu_0},0)$ is in its real Jordan normal form;}

\smallskip

\item[{\bf A2.}] {\it  Let $\al(\mu)\pm i \beta(\mu)$ be the pair of complex conjugated eigenvalues of $D_\bx\bg_1(\bx_\mu,\mu)$ such that $\al(\mu_0)=0,$ $\beta(\mu_0)=\omega_0>0.$ Assume that $\al'(\mu_0)\neq0$.}
\end{itemize}

Finally, define the number
\begin{equation}\label{l1TA}
\begin{array}{rl}
\ell_1&=\dfrac{1}{8}\left(\dfrac{\partial ^3\bg_1^1(\bx_{\mu_0},\mu_0)}{\partial x^3}+\dfrac{\partial ^3\bg_1^1(\bx_{\mu_0},\mu_0)}{\partial x\partial y^2}+\dfrac{\partial ^3\bg_1^2(\bx_{\mu_0},\mu_0)}{\partial x^2\partial y}+\dfrac{\partial ^3\bg_1^2(\bx_{\mu_0},\mu_0)}{\partial y^3}\right)\\
&+\dfrac{1}{8\omega_0}\left(\dfrac{\partial ^2\bg_1^1(\bx_{\mu_0},\mu_0)}{\partial x\partial y}\Big(\dfrac{\partial ^2\bg_1^1(\bx_{\mu_0},\mu_0)}{\partial x^2}+\dfrac{\partial ^2\bg_1^1(\bx_{\mu_0},\mu_0)}{\partial y^2}\Big)-\dfrac{\partial ^2\bg_1^2(\bx_{\mu_0},\mu_0)}{\partial x\partial y}\right.\\
&\Big(\dfrac{\partial ^2\bg_1^2(\bx_{\mu_0},\mu_0)}{\partial x^2}+\dfrac{\partial ^2\bg_1^2(\bx_{\mu_0},\mu_0)}{\partial y^2}\Big)-\dfrac{\partial ^2\bg_1^1(\bx_{\mu_0},\mu_0)}{\partial x^2}\dfrac{\partial ^2\bg_1^2(\bx_{\mu_0},\mu_0)}{\partial x^2}\\
&\left.+\dfrac{\partial ^2\bg_1^1(\bx_{\mu_0},\mu_0)}{\partial y^2}\dfrac{\partial ^2\bg_1^2(\bx_{\mu_0},\mu_0)}{\partial y^2}\right).
\end{array}
\end{equation}

We observe that hypothesis {\bf A1} implies the existence of a neighborhood $J_0\subset J$ of $\mu_0$ and $\e_1,$ $0<\e_1<\e_0$ such that, for every $(\mu,\e)\in J_0\times(-\e_1,\e_1),$ the differential equation \eqref{tm1} admits a unique $T$-periodic solution $\f(t,\mu,\e)$ satisfying $\f(0,\mu,\e)\to \bx_{\mu}$ as $\e\to0$ (see \cite[Lemma 1]{ITCanNov2018}). When  the differential equation \eqref{tm1} is defined in the extended phase space $\s^1\times\Omega,$  such a periodic solution is given by $\Phi(t,\mu,\e)=(t,\f(t,\mu,\e)).$

\begin{theorem}[\cite{ITCanNov2018}]\label{teo1} In addition to hypotheses {\bf A1} and {\bf A2}, assume that  $\ell_1\neq0.$ Then, for each $\e>0$ sufficiently small, there exists a $C^1$ curve $\mu(\e)\in J_0,$ with  $\mu(0)=\mu_0,$  neighborhoods $\U_{\e}\subset \s^1\times\Omega$ of the periodic solution $\Phi(t,\mu(\e),\e),$ and intervals  $J_{\e}\subset J_0$ of $\mu(\e)$ for which the following statements hold.
\begin{itemize}

\item[$(i)$] For $\mu\in J_{\e}$ such that $\ell_1(\mu-\mu(\e))\geq0,$ the periodic solution $\Phi(t,\mu(\e),\e)$ is unstable (resp. asymptotically stable), provided that $\ell_1>0$ (resp. $\ell_1<0$), and the differential equation \eqref{tm1} does not admit any invariant tori in $\U_{\e}.$

\item[$(ii)$] For $\mu\in J_{\e}$ such that $\ell_1(\mu-\mu(\e))<0,$ the differential equation \eqref{tm1} admits a unique invariant torus $T_{\mu,\e}$ in $\U_{\e}$ surrounding the periodic solution $\Phi(t,\mu,\e).$ Moreover,  $T_{\mu,\e}$ is unstable (resp. asymptotically stable), whereas the periodic solution $\Phi(t,\mu,\e)$ is asymptotically stable (resp. unstable), provided that $\ell_1>0$ (resp. $\ell_1<0$).

\item[$(iii)$] $T_{\mu,\e}$ is the unique invariant torus of the differential equation \eqref{tm1} bifurcating from the periodic solution $\Phi(t,\mu(\e),\e)$ in $\U_{\e}$ as $\mu$ passes through $\mu(\e).$
\end{itemize}
\end{theorem}
This result will be used for proving Theorem $1(c)$ (see Proposition \eqref{p5}).

\section{Proofs of Theorems \ref{tvd1} and \ref{tvd2}}\label{pr}
We are going to use the classical averaging method, Theorem \ref{CAT} for detecting periodic solutions stated in Theorem \ref{tvd1}$(a)$.

\begin{proposition}\label{p1}
Consider system \eqref{gvd} with coefficients \eqref{coefs} satisfying the hypotheses of Theorem \ref{tvd1}$(a)-(b)$. Then for $|\e|\neq 0$ sufficiently small, system \eqref{gvd} has three periodic solutions $\varphi_0(t,\e)$ and $\varphi_\pm(t,\e)$ satisfying $\varphi_0(t,\e)\to (0,0,0)$ and $\varphi_\pm(t,\e)\to (0,0,0)$ when $\e\to 0$.
\end{proposition}

\begin{proof}
Assume that the hypothesis in Theorem \ref{tvd1} hold. We will use the regular algorithm for applying averaging method. First we will use the following linear change of variables in order to have the linear part of system \eqref{gvd} in its real Jordan normal form

$$(x,y,z)=\left(\frac{h_0 \ov{Z}-\nu_0\ov{X}}{\omega },\ov{Y},\frac{h_0 k_0 \ov{Z}-\ov{X} \left(k_0 \nu_0+\omega ^2\right)}{h_0 \omega }\right)
.$$
In addition,  rescaling the system  by taking $(\ov X,\ov Y,\ov Z)=\e( X,  Y,  Z)$, we get
\begin{align*}
\dot{X}=&\,-\omega\, Y+\e \big(k_0 \mu_1 \nu_0 \omega ^2 (h_0 Z-\nu_0 X)+k_0 \nu_0 (\nu_0 X-h_0 Z)^3\\
&+Y \omega ^3 (k_0 \nu_1-\beta_1 h_0)\big)+\CO(\e^2),\\
\dot{Y}=&\, \omega\, X+\e\Bigg(\frac{(h_0 k_1-h_1 k_0) (h_0 Z-\nu_0 X)+h_1 X \omega ^2}{h_0 \omega }-\al_1 Y\Bigg)+\CO(\e^2),\\
\dot{Z}=&\frac{\e}{h_0 \omega ^4}\big(\nu_0 \omega ^2 (\nu_0 X-h_0 Z) \left(h_0^2 Z^2-2 h_0 \nu_0 X Z+\nu_0 \left(\nu_0 X^2-k_0 \mu_1\right)\right)\\
&+k_0 \nu_0^2 (\nu_0 X-h_0 Z)^3+\nu_0 Y \omega ^3 (k_0 \nu_1-\beta_1 h_0)+\mu_1 \nu_0 \omega ^4 (h_0 Z-\nu_0 X)\\
&+\nu_1 Y \omega ^5\big)+\CO(\e^2).
\end{align*}
Now, consider the cylindrical variables
\begin{equation}\label{cylin}
(X,Y,Z)=(r\cos \theta,r\sin \theta, z).
\end{equation}
Notice that $\dot \T=\omega+\CO(\e),$ which is positive for $\e>0$ sufficiently small. Therefore, taking  $\theta$  as the new time of the system,  system \eqref{gvd} becomes the following non-autonomous differential system
\begin{equation}\label{CAs}
\begin{split}
\dfrac{d r}{d \theta}=&\frac{k_0 \nu_0 \cos (\theta ) \left(h_0 \mu_1 \omega ^2 z-h_0^3 z^3\right)}{\omega ^5}+\frac{k_0 \nu_0^2 r \cos ^2(\theta ) \left(3 h_0^2 z^2-\mu_1 \omega ^2\right)}{\omega ^5}\\
&+\sin (\theta ) \Bigg(\frac{z (h_0 k_1-h_1 k_0)}{\omega ^2}+\frac{r \cos (\theta ) }{h_0 \omega ^2}\big(h_0 k_0 \nu_1-\beta_1 h_0^2-h_0 k_1 \nu_0+h_1 k_0 \nu_0\\
&+h_1 \omega ^2\big)\Bigg)-\frac{3 h_0 k_0 \nu_0^3 r^2 z \cos ^3(\theta )}{\omega ^5}+\frac{k_0 \nu_0^4 r^3 \cos ^4(\theta )}{\omega ^5}-\frac{\alpha_1 r \sin ^2(\theta )}{\omega }+\CO(\e^2)\vspace{0.1cm}\\
=&\e\,F^1_1(\theta,r,z)+\CO(\e^2),\vspace{0.3cm}\\
\dfrac{d z}{d \theta}=&\frac{\nu_0^2 r \cos (\theta ) \left(k_0 \nu_0+\omega ^2\right) \left(3 h_0^2 z^2-\mu_1 \omega ^2\right)}{h_0 \omega ^5}+\frac{\nu_0 z \left(k_0 \nu_0+\omega ^2\right) \left(\mu_1 \omega ^2-h_0^2 z^2\right)}{\omega ^5}\\
&+\frac{\nu_0^4 r^3 \cos ^3(\theta ) \left(k_0 \nu_0+\omega ^2\right)}{h_0 \omega ^5}+\frac{r \sin (\theta ) \left(\nu_1 \left(k_0 \nu_0+\omega ^2\right)-\beta_1 h_0 \nu_0\right)}{h_0 \omega ^2}\\
&-\frac{3 \nu_0^3 r^2 z \cos ^2(\theta ) \left(k_0 \nu_0+\omega ^2\right)}{\omega ^5}+\CO(\e^2)\vspace{0.1cm}\\
=&\e\,F^2_1(\theta,r,z)+\CO(\e^2).
\end{split}
\end{equation}
Observe that the non-autonomous differential system \eqref{CAs} is  written in the standard form \eqref{tm13} for applying the averaging theorem. Thus, identifying
\[
t=\theta, \ T=2\pi, \ \bz=(r,z), \text{ and } \bF_1(\theta,r,z)=\Big(F^1_1(\theta,r,z),F_1^2(\theta,r,z)\Big),
\]
we compute the first-order averaged function \eqref{avf}, $\bg_1(r,z)=\big(\bg_1^1(r,z),\bg_1^2(r,z)\big)$, as
\begin{equation}\label{avfCA}
\begin{split}
\bg_1^1(r,z)=& r \left(\frac{3 h_0^2 k_0 \nu_0^2 z^2}{2 \omega ^5}-\frac{\al_1 \omega ^2+k_0 \nu_1 \nu_0^2}{2 \omega ^3}\right)+\frac{3 k_0 \nu_0^4 r^3}{8 \omega ^5},\vspace{0.2cm}\\
\bg_1^2(r,z)=&-\frac{h_0^2 \nu_0 z^3 \left(k_0 \nu_0+\omega ^2\right)}{\omega ^5}-\frac{3 \nu_0^3 r^2 z \left(k_0 \nu_0+\omega ^2\right)}{2 \omega ^5}+\frac{\nu_1 \nu_0 z \left(k_0 \nu_0+\omega ^2\right)}{\omega ^3}.
\end{split}
\end{equation}
We are interested in the solutions of the non-linear system $\bg_1(r,z) =(0,0)$ with $r>0$. In this case, since $\delta_a>0,$ $\delta_b>0,$ and $\delta_c>0$, there are three solutions $(r_0,0)$, $(r_1,z_1)$ and $(r_1,z_2)$, namely
\begin{equation*}
\begin{split}
r_0=\frac{2\omega }{\nu_0^2}\sqrt{\frac{\al_1 \omega ^2+k_0\mu_1 \nu_0^2 }{3 k_0}},&\quad
r_1=\frac{2\omega }{\nu_0^2} \sqrt{\frac{2 k_0\mu_1 \nu_0^2 -\al_1 \omega ^2}{15 k_0}},\\
z_1=\frac{\omega}{h_0 \nu_0}\sqrt{\frac{2 \al_1 \omega ^2+k_0\mu_1 \nu_0^2}{5 k_0}},&\quad
z_2=-\frac{\omega}{h_0 \nu_0}  \sqrt{\frac{2 \al_1 \omega ^2+k_0\mu_1 \nu_0^2}{5 k_0}}.
\end{split}
\end{equation*}

 Moreover, the Jacobian determinant of $\bg_1$ at $(r_0,0)$ and $(r_1,z_i)$ for $i=1,2$ is given respectively by
\[
\det\left(r_0,0)\right)=-\dfrac{\delta_d \delta_c \delta_c}{k_0^3 \al_1 \nu_0 \omega^6}=\lambda_1\lambda_2 \quad \mbox{and} \quad
\det\left(r_1,z_i)\right)=\frac{2 \delta_a \delta_b \delta_d}{5 k_0^3 \nu_0 \al_1\omega ^6}=\lambda_-\lambda_+,
\]
and from our assumptions we have that $\det\left(r_0,0)\right)\neq 0$ and $\det\left(r_1,z_i)\right)\neq 0$. Thus, the result follows by applying Theorem  \ref{CAT} and going back through the changes of variables \eqref{cylin}.
\end{proof}

The bifurcation of invariant tori  stated in Theorem \ref{tvd1}$(c)$
will be proved using Theorem \ref{teo1}.

\begin{proposition}\label{p5} Consider $\ell_1$ as defined in \eqref{l1TA} and assume that system \eqref{gvd} satisfies the hypotheses of Proposition \ref{p1}. If $\ell_1\neq 0$, then there exists a smooth curve $\mu(\e),$ defined for $\e>0$ sufficiently small and satisfying
$$
\widehat{\mu}_1=-\frac{\al_1 \left(5 k_0 \nu_0 +4 \omega ^2\right)}{2 k_0 \nu_0 ^2},
$$
and intervals $J_{\e}$ containing $\mu(\e)$ such that a unique invariant torus bifurcates from the periodic solution $\varphi(t,\mu(\e),\e)$ as $\mu_1$ passes through $\mu(\e).$ Such a torus exists whenever $\mu_1\in J_{\e}$ and $\ell_1(\mu_1-\mu(\e))<0,$ and surrounds the periodic solution $\f(t,\mu_1,\e).$ In addition,  if $\ell_1>0$ (resp. $\ell_1<0$) such a torus is unstable (resp. asymptotically stable), whereas the periodic solution $\varphi(t,\mu_1,\e)$ is asymptotically stable (resp. unstable).
\end{proposition}

\begin{proof}
Consider the first-order averaging function $\bg_1$, defined in \eqref{avfCA}, of the non-autonomous differential system \eqref{CAs}.  The proof will be basically, putting the truncated averaging system
$$
\dot \bx=\e\,\bg_1(\bx),
$$
into the correct  coordinates in order to check hypotheses {\bf A1} and {\bf A2} and applying Theorem \ref{teo1}. We point out that we will proof our result for  $\by_{\mu_1}=(r_1,z_1)$. We observe that the case $\by_{\mu_1}=(r_1,z_2)$ can be done analogously.

First, we have to write $D \bg_1(\by_{\widehat{\mu}_1},\widehat{\mu}_1)$ in its normal Jordan form. Accordingly,  consider the linear change of variables
$$
(r,z)=T\cdot(u,v)=\left(\frac{ h_0 \omega ^2 }{\nu_0 } \left(\sqrt{5} u+v\right) \sqrt{\frac{\al_1 k_0 \nu_0 }{6 \omega ^2 \delta_d}},v\right).
$$
Denoting $\bx_{\mu_1}=T^{-1}\,\by_{\mu_1}$ we have
\begin{align*}
\bx_{\mu_1}&=\frac{1}{h_0 \nu_0 }\Bigg(\frac{1}{5 }\sqrt{\frac{\delta_b}{k_0^2}}+\frac{2  \nu_0}{5  k_0} \sqrt{\frac{2 \delta_d \delta_a}{k_0\al_1\nu_0 ^3}},-\sqrt{\frac{\delta_b}{5 k_0^2}}\Bigg),
\end{align*}
 and we see that the Jacobian matrix $D\tilde \bg_1(\bx_{\widehat{\mu}_1},\widehat{\mu}_1)$  is in its normal Jordan form.
Moreover, taking $\bx=(u,v)$, we get
\begin{equation*}\label{npromed}
\begin{split}
\tilde\bg_1^1(\bx,\mu_1)&=\frac{\nu_0 }{16 \sqrt{5} \omega ^3} \Bigg(\frac{h_0^2}{k \nu_0 +\omega ^2} \big(5 k_0^2 \nu_0 ^2 \big(\sqrt{5} u^3+7 u^2 v+7 \sqrt{5} u v^2+9 v^3\big)+4 k_0 \\
&. \nu_0  v \omega ^2\big(5 u^2+8 \sqrt{5} u v+15 v^2\big)+16 v^3 \omega ^4\big)-8 k_0 \mu_1\nu_0  \big(\sqrt{5} u+3 v\big)\\
&-\frac{8 \omega ^2}{\nu_0} \big(\sqrt{5} \al_1 u+v (\al_1+2 \mu_1\nu_0)\big)\Bigg),\\
\tilde\bg_1^2(\bx,\mu_1)&=-\frac{\nu_0  v }{4 \omega ^3}\Big(h_0^2 \left(k_0 \nu_0  \left(5 u^2+2 \sqrt{5} u v+5 v^2\right)+4 v^2 \omega ^2\right)\\
&-4 \mu_1\left(k_0 \nu_0 +\omega ^2\right)\Big).
\end{split}
\end{equation*}
Now, in order to check hypothesis {\bf A1}, let
\[
 \omega_0=\frac{\sqrt{5}\, \delta_d}{\omega ^3}.
 \]
We compute the characteristic polynomial of the Jacobian matrix $D\tilde\bg_1(\bx_{\widehat{\mu}_1},\widehat{\mu}_1),$ obtaining
\[
\begin{split}\label{charp}
p(\lambda) =&
\lambda ^2+\frac{\lambda }{5 \omega } \left(\al_1 \left(\frac{4 \omega ^2}{k_0 \nu }+5\right)+2 \mu_1 \nu_0 \right)+\dfrac{2 \delta_a \delta_b \delta_d}{5 \al_1 k_0^3 \nu_0  \omega ^6},
\end{split}
\]
where $\delta_a$, $\delta_b$ and $\delta_d$ are defined in \eqref{param}.  Denoting the roots of $p(\lambda)$ by $\lambda(\mu_1)=a (\mu_1)\pm i\, b(\mu_1)$, it is  straightforward to see that $a(\widehat{\mu}_1)=0$ and $b(\widehat{\mu}_1)=\omega_0$. So, hypothesis {\bf A1} holds.

In order to verify hypothesis {\bf A2}, we compute the derivative of the real part of the eigenvalues of \eqref{charp} at $\mu_1=\widehat{\mu}_1$, and we obtain the relation
$$
\dfrac{d \,a(\widehat{\mu}_1)}{d \mu_1}=-\frac{\nu_0 }{5 \omega }\neq 0,
$$
which verifies hypothesis {\bf A2}.  Finally, taking \eqref{l1TA} into account, we compute
\begin{align*}
\ell_{1}=&-\frac{3 \pi  h_0^2 \nu_0  \left(5 k_0^2 \nu_0 ^2+16 k_0 \nu_0  \omega ^2+8 \omega ^4\right)}{8 \omega ^3 \left(k_0 \nu_0 +\omega ^2\right)}\neq 0.
\end{align*}
Hence, we conclude the proof of Theorem \ref{tvd1}$(c)$ by applying Theorem \ref{teo1}.
\end{proof}

The classical averaging method, Theorem \ref{CAT}, will be used for proving the next two propositions.
\begin{proposition}\label{p2}
Assume that system \eqref{gvd} with coefficients in \eqref{coefs} satisfies  hypothesis {\bf H}$_1$. Then, for $|\e|\neq 0$ sufficiently small,  system \eqref{gvd}  has a periodic solution $\varphi(t,\e)$ satisfying $\varphi(t,\e)\to (0,0,0)$ when $\e\to 0$.
\end{proposition}

\begin{proof}
Following the same method used in the proof of \eqref{p1}, the first step is to write  system \eqref{gvd}  in the standard form \eqref{tm13} in order to use the averaging theory for detecting its periodic solutions. So we consider the linear change of variables
$$(x,y,z)=\left(\ov X ,-\dfrac{\ov X \mu_0\nu_0+\ov Y \omega}{\nu_0},\dfrac{\ov Y \beta_0}{\nu_0}-\dfrac{\ov Y \beta_0 \mu_0+\ov Z \nu_0}{\omega} \right).$$
In addition, taking $(\ov  X, \ov  Y, \ov Z)=\sqrt{\e}(X,Y,Z)$, we see that the unperturbed system (that is, system with $\e=0$) in these new variables can be written as $\big(\dot X,\dot Y,\dot Z\big)=\big(- \omega Y,\omega X,0\big)$. Thus, using cylindrical coordinates $(X,Y,Z)=(r\cos\theta,r\sin \theta,z)$, we see that $\dot \T=\omega+\CO(\e),$ which is positive for $|\e|$ sufficiently small. Therefore, we take $\theta$ as the new time of the system so that  system \eqref{gvd} becomes a $T$-periodic non-autonomous differential system.

Observe that the non-autonomous differential system \eqref{s22} is  written in the standard form \eqref{tm13} for applying the averaging theorem. Thus, identifying
\[
t=\theta, \ T=2\pi, \ \bz=(r,z),  \text{ and } \bF_{i}(\theta,r,z)=\left( F_{i}^1(\theta,r,z), F_{i}^2(\theta,r,z)\right) \text{ for } i=1,2,
\]
we compute the first order averaged function  obtaining the first averaged function
$$
\bg_1(r,z)=\Bigg(\frac{r \big(\omega ^2 (\mu_0 \nu_1-\al_1+\mu_1 \nu_0)-\beta_0 h_1 \mu_0 \nu_0\big)}{2 \omega ^3}-\frac{3 \nu_0 r^3}{8 \omega },\frac{\beta_0 h_1 \mu_0 \nu_0 z}{\omega ^3}\Bigg).
$$
This function has the following simple zero
$$
(\ov r,\ov z)=\left(\frac{2}{\omega } \sqrt{\frac{\omega ^2 (\mu_0 \nu_1-\al_1+\mu_1 \nu_0)-\beta_0 h_1 \mu_0 \nu_0}{3 \nu_0}},0\right).
$$
In fact, we have the following Jacobian determinant
$$
|D\bg_1(\ov r,\ov z)|=\frac{\beta_0 h_1 \mu_0 \nu_0 \left(\omega ^2 (\al_1-\mu_0 \nu_1-\mu_1 \nu_0)+\beta_0 h_1 \mu_0 \nu_0\right)}{\omega ^6},
$$
which is non-zero by hypothesis. The proof of  the proposition follows from Theorem \ref{CAT} and going back through the changes of variables.
\end{proof}

\begin{proposition}\label{p3}Assume that system \eqref{gvd} with coefficients in \eqref{coefs} satisfies the hypothesis {\bf H}$_2$. Then, for $|\e|\neq 0$ sufficiently small,  system \eqref{gvd}  has a periodic solution $\varphi(t,\e)$ satisfying $\varphi(t,\e)\to (0,0,0)$ when $\e\to 0$.
\end{proposition}
\begin{proof}
Again we perform the linear change of variables
$$
(x,y,z)=\left(\ov X + \dfrac{h_0 \ov Z \nu_0^2}{\omega^3},-\ov X \mu_0-\dfrac{h_0\ov Z\mu_0 \nu_0^2}{\omega^3}-\dfrac{\ov Y \omega}{\nu_0},-\dfrac{\ov Z \nu_0}{\omega}\right),
$$
in order to put the linear part of system \eqref{gvd} in its Jordan normal form. In addition, we take $(\ov  X, \ov  Y, \ov Z)=\sqrt{\e}(X,Y,Z)$, and the unperturbed system becomes $\big(\dot X,\dot Y,\dot Z\big)=\big(- \omega Y,\omega X,0\big)$. Thus, taking cylindrical coordinates $(X,Y,Z)=(r\cos\theta,r\sin \theta,z)$, we have $\dot \T=\omega+\CO(\e),$. Therefore, we can take $\theta$ as the new time of the system so that  system \eqref{gvd} becomes the following non-autonomous differential system
\begin{equation*}
\dfrac{d r}{d \theta}=\e F_{1}^1(\theta,r,z)+\CO(\e^2),\quad
\dfrac{d z}{d\theta}=\e F_{1}^2(\theta,r,z)+\CO(\e^2),
\end{equation*}
 where $(\theta,r,z)\in\R\times\R_+\times \R$. Identifying
\[
t=\theta, \ T=2\pi, \ \bz=(r,z),  \text{ and } \bF_{1}(\theta,r,z)=\left( F_{1}^1(\theta,r,z), F_{1}^2(\theta,r,z)\right),
\]
we have that this non-autonomous differential system is in the standard form \eqref{tm13} for applying the averaging theorem. Thus,
we compute the first-order averaged function \eqref{avf} obtaining $\bg_1(r,z)=\big(g^1_1(r,z),g^2_1(r,z)\big)$ where
\begin{align*}
g_1^1(r,z)=&r \Bigg(\frac{\omega ^2 (\mu_0 \nu_1-\al_1+\mu_1 \nu_0)-\beta_1 h_0 \mu_0 \nu_0}{2 \omega ^3}-\frac{3 h_0^2 \nu_0^5 z^2}{2 \omega ^7}\Bigg)-\frac{3 \nu_0 r^3}{8 \omega },\\
g_1^2(r,z)=&\frac{\beta_1 h_0 \mu_0 \nu_0 z}{\omega ^3}.\\
\end{align*}
This function has the following simple zero
$$
(\ov r,\ov z)=\left(\frac{2}{\omega } \sqrt{\frac{\omega ^2 (\mu_0 \nu_1-\al_1+\mu_1 \nu_0)-\beta_1 h_0 \mu_0 \nu_0}{3 \nu_0}},0 \right).
$$
In fact, we have the following Jacobian determinant
$$
|D\bg_1(\ov r,\ov z)|=\frac{\beta_1 h_0 \mu_0 \nu_0 \left(\omega ^2 (\al_1-\mu_0 \nu_0-\mu_1 \nu_0)+\beta_1 h_0 \mu_0 \nu_0\right)}{\omega ^6}
$$
which is non-zero by hypothesis. The proof of the proposition follows from Theorem \ref{CAT} and going back through the changes of variables.
\end{proof}
In order to be able to detect a periodic solution of system \eqref{gvd} under the hypothesis {\bf H}$_3$ we need to apply Theorem \ref{LSt1}, that combines the Brouwer degree with the Lyapunov-Schimidt reduction.

\begin{proposition}\label{p4} Assume that system \eqref{gvd} with coefficients in \eqref{coefs} satisfies the hypothesis {\bf H}$_3$. Then, for $|\e|\neq 0$ sufficiently small,  system \eqref{gvd}  has a periodic solution $\varphi(t,\e)$ satisfying $\varphi(t,\e)\to (0,0,0)$ when $\e\to 0$.
\end{proposition}
\begin{proof}
We start by writing the linear part of system \eqref{gvd} in its Jordan normal form, so consider the linear change of variables
$$(x,y,z)=\left(\ov Z\dfrac{h_0}{\omega},\ov Y,\dfrac{h_0 \ov Z}{\omega}-\dfrac{\ov X \omega}{h_0} \right).$$
In addition, taking $(\ov  X, \ov  Y, \ov Z)=\e(X,Y,Z)$, we see that the unperturbed system (that is, the system with $\e=0$)  can be written as $\big(\dot X,\dot Y,\dot Z\big)=\big(- \omega Y,\omega X,0\big)$. Thus, taking cylindrical coordinates $(X,Y,Z)=(r\cos\theta,r\sin \theta,z)$, we see that $\dot \T=\omega+\CO(\e),$ which is positive for $|\e|$ sufficiently small. Therefore, by doing a time-rescaling, $\theta$ can be taken as the new time of the system so that  system \eqref{gvd} becomes the following non-autonomous differential system
\begin{equation}\label{s22}
\dfrac{d r}{d \theta}=\sum_{i=1}^4\e^i F_{i}^1(\theta,r,z)+\CO(\e^5),\quad
\dfrac{d z}{d\theta}=\sum_{i=1}^4\e^i F_{i}^2(\theta,r,z)+\CO(\e^5)
\end{equation}
 where $(\theta,r,z)\in\R\times\R_+\times \R$. Due to the extension of the expressions of $F_i^j(\theta,r,z)$, $i=1,\dots,4$ and $j=1,2$, we shall omit them here. However, they are trivially computed in terms of the parameters $\omega$, $h_i,\al_i,\mu_i, k_i,\beta_i$, and $\nu_i.$

Observe that the non-autonomous differential system \eqref{s22} is  written in the standard form \eqref{tm13} for applying the averaging theorem. Thus, identifying
\[
t=\theta, \ T=2\pi, \ \bz=(r,z),  \text{ and } \bF_{i}(\theta,r,z)=\left( F_{i}^1(\theta,r,z), F_{i}^2(\theta,r,z)\right),
\]
we compute the first-order averaged function \eqref{avf} obtaining $\bg_1(r,z)=\left(0,0\right).$  The second averaged function is
\begin{equation*}
\bg_2(r,z)=\left(0,\frac{2 \pi  \nu_1 \left(\mu_1 \omega ^2 -h_0^2 z^2\right)z}{\omega ^3}\right),
\end{equation*}
for $(r,z)\in\R_+\times\R.$ This function has three continuum of zeros, namely
$$
\CZ_0=\left\lbrace
\bz_{r_0}=\left(r,0\right):r>0
\right\rbrace \quad \mbox{and} \quad \CZ_\pm=\left\lbrace
\bz_{r_\pm}=\left(r,\pm\frac{\sqrt{\mu_1} \omega }{h_0}\right):r>0
\right\rbrace.
$$
Moreover, the Jacobian matrix of $\bg_2$ can be written as
$$
D{\bg_2}(r,z)=\left(
\begin{array}{cc}
 0 & 0 \\
 0 & \frac{2 \pi  \nu_1 }{\omega ^3}\left(\mu_1 \omega ^2-3 h_0^2 z^2\right) \\
\end{array}
\right).
$$
In order to apply Theorem \ref{LSt1}, we compute the third and fourth order averaged functions using the formulas  \eqref{avf}, obtaining $\bg_3(r,z)=\big(g^1_3(r,z),g^2_3(r,z)\big)$ and $\bg_4(r,z)=\big(g^1_4(r,z),g^2_4(r,z)\big)$, as
\begin{align*}
g_3^1(r,z)=&\frac{\pi}{\omega ^5} \Big(h_0^2 \nu_1 z^2 (2 h_0 k_1 z-2 h_1 k_0 z+3 k_0 \nu_1 r)-\mu_1 \nu_1 \omega ^2 (2 h_0 k_1 z-2 h_1 k_0 z\\
&+k_0 \nu_1 r)-\al_3 r \omega ^4\Big),
\end{align*}

\begin{align*}
g_3^2(r,z)=&\frac{\pi }{h_0 \omega ^5} \Big(h_0^3 \nu_1 z^3 (\beta_1 h_0-3 k_0 \nu_1)+h_0 \omega ^2 z \big(-2 h_0^2 \nu_2 z^2+h_0 \nu_1 \big(h_1 z^2-\beta_1 \mu_1\\
&-6 \nu_1 r z\big)+3 k_0 \mu_1 \nu_1^2\big)+\omega ^4 \big(2 h_0 z (\mu_1 \nu_2+\mu_2 \nu_1)-h_1 \mu_1 \nu_1 z+2 \mu_1 \nu_1^2 r\big)\Big),
\end{align*}
\begin{align*}
g_4^1(r,z)=&\frac{\pi}{2 h_0\omega ^7}  \Big(h_0^3 \nu_1 z^2 \big(-3 k_0 \nu_1 (-4 h_0 k_1 z+3 \beta_1 h_0 r+4 h_1 k_0 z)+2 \beta_1 h_0 z (h_1 k_0\\
&-h_0 k_1)+9 k_0^2 \nu_1^2 r\big)+h_0 \omega ^2 \big(4 h_0^2 \nu_2 z^2 (h_0 k_1 z-h_1 k_0 z+3 k_0 \nu_1 r)+\nu_1\\
&.\big(2 h_0 z^3 \big(-3 h_0 h_1 k_1+2 h_0 (h_0 k_2-h_2 k_0)+3 h_1^2 k_0\big)+3 h_0 \nu_1 r z^2 (6 h_0 k_1\\
&-7 h_1 k_0)+k_0 \mu_1 \nu_1 (-8 h_0 k_1 z+3 \beta_1 h_0 r+8 h_1 k_0 z)+2 \beta_1 h_0 \mu_1 z (h_0 k_1\\
&-h_1 k_0)-3 k_0 \nu_1^2 r (k_0 \mu_1-4 h_0 r z)\big)\big)-\omega ^4 \big(h_0^2 (2 k_1 z (\al_3+2 \mu_1 \nu_2\\
&+2 \mu_2 \nu_1)+4 k_2 \mu_1 \nu_1 z+\al_3 \beta_1 (-r))+h_0 (\nu_1 (-4 h_2 k_0 \mu_1 z+k_0 r (\al_3\\
&+4 \mu_1 \nu_2+2 \mu_2 \nu_1)+6 k_1 \mu_1 \nu_1 r)-2 h_1 z (k_0 (\al_3+2 \mu_1 \nu_2+2 \mu_2 \nu_1)\\
&+3 k_1 \mu_1 \nu_1))+h_1 k_0 \mu_1 \nu_1 (6 h_1 z-7 \nu_1 r)\big)+r \omega ^6 (\al_3 h_1-2 \al_4 h_0)\Big),
\end{align*}
\begin{align*}
g_4^2(r,z)=&\frac{\pi}{4 h_0^2 \omega ^7}  \Big(24 \pi  h_0^6 \nu_1^2 \omega  z^5-3 h_0^4 \nu_1 z^3 (\beta_1 h_0-k_0 \nu_1) (\beta_1 h_0-5 k_0 \nu_1)-32 \pi  h_0^4\\
& .\mu_1 \nu_1^2 \omega ^3 z^3+\omega ^6 \big(8 h_0^2 z (\mu_1 \nu_3+\mu_2 \nu_2+\mu_3 \nu_1)-4 h_0 z (h_1 \mu_1 \nu_2+h_1 \mu_2 \nu_1\\
&+h_2 \mu_1 \nu_1)+4 h_0 \nu_1 r (\al_3+4 \mu_1 \nu_2+2 \mu_2 \nu_1)+h_1 \mu_1 \nu_1 (3 h_1 z-4 \nu_1 r)\big)\\
&+8 \pi  h_0^2 \mu_1^2 \nu_1^2 \omega ^5 z+h_0 \omega ^4 \big(-8 h_0^3 \nu_3 z^3+4 h_0^2 z \big(-\nu_1 (\beta_1 \mu_2+\beta_2 \mu_1)\\
&+h_1 \nu_2 z^2+h_2 \nu_1 z^2-\nu_2 (\beta_1 \mu_1+12 \nu_1 r z)\big)+h_0 \nu_1 \big(-3 h_1^2 z^3+2 h_1 z\\
&. (\beta_1 \mu_1+6 \nu_1 r z)+4 \nu_1 (3 k_0 \mu_2 z+5 k_1 \mu_1 z-3 r (\beta_1 \mu_1+3 \nu_1 r z))+24 k_0\\
&. \mu_1 \nu_2 z\big)+2 k_0 \mu_1 \nu_1^2 (12 \nu_1 r-13 h_1 z)\big)+h_0^2 \omega ^2 z \big(4 h_0^3 z^2 (\beta_1 \nu_2+\beta_2 \nu_1)\\
&+h_0^2 \nu_1 \big(3 \beta_1^2 \mu_1-2 \beta_1  z^2-12 z^2 (2 k_0 \nu_2+3 k_1 \nu_1)+36 \beta_1 \nu_1 r z\big)-6 h_0 k_0 \\
&.\nu_1^2 \big(3 \beta_1 \mu_1-7 h_1 z^2+12 \nu_1 r z\big)+15 k_0^2 \mu_1 \nu_1^3\big)\Big).
\end{align*}

We are going to find a periodic solution bifurcating from
$$
\CZ_+=\left\lbrace
\bz_{r_+}=\left(r,\frac{\sqrt{\mu_1} \omega }{h_0}\right):r>0
\right\rbrace.
$$
Thus we use $\bz_{r_+}$ to compute the coefficients of the bifurcation  function \eqref{cF} obtaining
\begin{align*}
f_1(r)=&\frac{\pi  r \big(2 k_0\mu_1 \nu_1^2-\al_3 \omega ^2\big)}{\omega ^3},\\
f_2(r)=&\frac{\pi  r}{2 h_0\omega ^5} \Big(\omega ^2 \big(\al_3 \beta_1 ^2+h_0\nu_1 (k_0(-\al_3+8 \mu_1 \nu_2+4 \mu_2 \nu_1)+4 k_1 \mu_1 \nu_1)\\
&-6 h_1 k_0\mu_1 \nu_1^2\big)+\omega ^4 (\al_3 h_1-2 \al_4 h_0)+6 h_0k_0\mu_1 \nu_1^2 (k_0\nu_1-\beta_1 h_0)\Big)\\
&+\frac{\pi  \sqrt{\mu_1} (h_0k_1-h_1 k_0) \big(2 k_0\mu_1 \nu_1^2-\al_3 \omega ^2\big)}{h_0\omega ^4},
\end{align*}
having $ \Delta_r=\dfrac{4 \pi  \mu_1 \nu_1}{\omega }.$ Now we compute the zeroes of the function
\begin{equation*}
\CF^2(r)=\e f_1(r)+\e^2 f_2(r),
\end{equation*}
and we obtain
$$
a_\e=\dfrac{-2   \omega\,\e  (h_1 k_0-h_0 k_1) \left(2 k_0 \mu_1 \nu_1^2-\al_3 \omega ^2\right)\sqrt{\mu_1}}{h_0 \left(2 \al_3 \omega ^4-4 k_0 \mu_1 \nu_1^2 \omega ^2\right)+\e \rho(h_0,h_1,k_0,k_1,\nu_1,\nu_2,\mu_1,\mu_2,\al_3,\al_4,\omega)},
$$
where
\begin{align*}
&\rho(h_0,h_1,k_0,k_1,\nu_1,\nu_2,\mu_1,\mu_2,\al_3,\al_4,\omega)=\omega ^4 (2 \al_4 h_0-\al_3 h_1)+6 h_0 k_0 \mu_1 \nu_1^2 (\beta_1 h_0\\
&-k_0 \nu_1)  -\omega ^2 \big(\al_3 \beta_1 h_0^2+h_0 \nu_1 (k_0 (8 \mu_1 \nu_2-\al_3+4 \mu_2 \nu_1)+4 k_1 \mu_1 \nu_1)-6 h_1 k_0 \mu_1 \nu_1^2\big).
\end{align*}
Then, under these conditionsthe hypothesis $(i)$, $(ii)$ and $(iii)$ of Theorem \ref{LSt1} hold. The hypothesis $(iv)$ also holds since we have
\begin{align*}
|\partial_\al\CF^2(a_\e,\e)|&\geq \e \Bigg|\left|\frac{\pi  \left(2 k_0 \mu_1 \nu^2-\al_3 \omega ^2\right)}{\omega ^3}\right|-\e \Big|\kappa(h_0,h_1,k_0,k_1,\mu_1,\mu_2,\nu_1,\nu_2,\al_4,\omega)\Big|\Bigg|
\end{align*}
where $\kappa(h_0,h_1,k_0,k_1,\mu_1,\mu_2,\nu_1,\nu_2,\al_4,\omega)=$
\begin{align*}
& \frac{\pi }{2 h_0 \omega ^5}  \Big(\omega ^2 \big(\al_3 \beta_1 h_0^2+h_0 \nu_1 (k_0 (8 \mu_1 \nu_2-\al_3+4 \mu_2 \nu_1)+4 k_1 \mu_1 \nu_1)-6 h_1 k_0 \mu_1 \nu_1^2\big)\\
&+\omega ^4 (\al_3 h_1-2 \al_4 h_0)+6 h_0 k_0 \mu_1 \nu_1^2 (k_0 \nu_1-\beta_1 h_0)\Big).
\end{align*}
Thus for $\e>0$ sufficiently small we can find $P_0>0$ satisfying
$$
|\partial_\al\CF^2(a_\e,\e)|\geq \e\, P_0.
$$
The proof of the proposition follows from applying Theorem \ref{LSt1} and going back through the changes of coordinates. The same result will be obtained in $\CZ_-$, although both periodic orbits cannot coexist.  For $\CZ_0$ the result is not conclusive.
\end{proof}

The next result will be obtained by direct  analysis of the Poincar\'e map of the system \eqref{gvd} under the hypothesis {\bf H}$_4$.

\begin{proposition}\label{p5}Assume that system \eqref{gvd} in coefficients \eqref{coefs} satisfies the hypothesis {\bf H}$_4$. Then, for $|\e|\neq 0$ sufficiently small,  system \eqref{gvd}  has a periodic solution $\varphi(t,\e)$ satisfying $\varphi(t,\e)\to (0,0,0)$ when $\e\to 0$.
\end{proposition}

\begin{proof}
Using the linear change of variables
$$
(x,y,z)=\left(\ov X, -\dfrac{\ov Y \omega}{\nu_0}, \dfrac{\ov X \beta_0}{\nu_0}-\dfrac{\ov Z \nu_0}{\omega} \right),
$$
and taking $(\ov  X, \ov  Y, \ov Z)=\sqrt{\e}(X,Y,Z)$, the unperturbed part of system \eqref{gvd} becomes $\big(\dot X,\dot Y,\dot Z\big)=\big(- \omega Y,\omega X,0\big)$. Thus, we use cylindrical coordinates $(X,Y,Z)=(r\cos\theta,r\sin \theta,z)$, and $\dot \T=\omega+\CO(\e),$ for taking  $\theta$  as the new time of the system so that  system \eqref{gvd} becomes the following non-autonomous differential system
\begin{equation*}
\dfrac{d r}{d \theta}=\sum_{i=1}^2\e^i F_{i}^1(\theta,r,z)+\CO(\e^3),\quad
\dfrac{d z}{d\theta}=\sum_{i=1}^2\e^i F_{i}^2(\theta,r,z)+\CO(\e^3),
\end{equation*}
 where $(\theta,r,z)\in\R\times\R_+\times \R$. Again we omit the expressions of $F_i^j(\theta,r,z)$, $i=1,2$ because they are very large.

The non-autonomous differential system is in the standard form \eqref{tm13}. Thus, identifying
\[
t=\theta, \ T=2\pi, \ \bz=(r,z),  \text{ and } \bF_{i}(\theta,r,z)=\left( F_{i}^1(\theta,r,z), F_{i}^2(\theta,r,z)\right) \text{ for } i=1,2,
\]
we compute the first-order averaged function \eqref{avf} obtaining
$$
\bg_1(r,z)=\left(-\frac{3 \nu_0 r^3}{8 \omega }-\frac{r (\al_1-\mu_1 \nu_0)}{2 \omega },0\right).
$$
for $(r,z)\in\R_+\times\R.$ This function has a continuum of zeroes with positive $r$, namely
$$
\CZ_+=\left\lbrace
\bz=\left(2 \sqrt{\frac{\mu_1\nu_0-\al_1}{3 \nu_0}},z\right):z\in \R
\right\rbrace.
$$
Moreover, the Jacobian matrix of $\bg_1$ can be written as
$$
D{\bg_1}(r,z)=\left(
\begin{array}{cc}
-\frac{1}{8 \omega }\big(4 \al_1-4 \mu_1 \nu_0+9 \nu_0 r^2\big) & 0 \\
 0 & 0 \\
\end{array}
\right).
$$
We are going to find a fixed point for the Poincar\'e map
$$
\Pi(\bz,\e)=\bz+\e\bg_1(\bz) +\CO(\e^{2}).
$$
As the first averaged function has no simple zero, we compute the second order averaged function $\bg_2(r,z)=\big(g^1_2(r,z),g^2_2(r,z)\big)$ where
\begin{align*}
g_2^1(r,z)=&\frac{\pi }{32 \nu_0 \omega ^4} \big(4 \beta_0 h_1 \nu_0 r \omega  \big(4 \al_1-12 \mu_1\nu_0+9 \nu_0 r^2\big)+8 h_1 \nu_0^3 z \big(4 \al_1-4 \mu_1\nu_0\\
&+9 \nu_0 r^2\big)+r \omega  \big(\omega  \big(4 \omega  \big(4 \al_1 \nu_1+\nu_0 \big(-8\al_2+4 \mu_1\nu_1+8 \mu_2 \nu_0-3 \nu_1 r^2\big)\big)\\
&+\pi  \nu_0 \big(4 \al_1-4 \mu_1\nu_0+3 \nu_0 r^2\big) \big(4 \al_1-4 \mu_1\nu_0+9 \nu_0 r^2\big)\big)-4 \text{k1} \nu_0^2 \big(4 \al_1\\
&-4 \mu_1\nu_0+3 \nu_0 r^2\big)\big)\big),\\
g_2^2(r,z)=&\frac{\pi  \beta_0 h_1 \nu_0 z \big(2 \mu_1-3 r^2\big)}{\omega ^3}+\frac{\pi  r (\beta_1 \nu_0-\beta_0 \nu_1) \big(4 \al_1-4 \mu_1\nu_0+3 \nu_0 r^2\big)}{4 \nu_0^3}.
\end{align*}
Thus, we use the Poincar\'e map  to define the displacement map
\begin{align*}
d((r,z),\e)=&\Pi((r,z),\e)-(r,z)\\
&=\e\bg_1\big(r,z\big)+\e^2\bg_2\big(r,z\big) +\CO(\e^{3}).
\end{align*}
We consider now the function $\xi \colon \R^2\times(-\e_0,\e_0)\mapsto \R^2$,
$$
\xi((r,z),\e)=\left(2 \sqrt{\frac{\mu_1\nu_0-\al_1}{3 \nu_0}},z\right)+\e(r,0),
$$
and composing it with the displacement map we obtain
\begin{equation*}
\begin{split}
H\big((r,z),\e\big)=&d(\xi(r,z),\e)=\e H_1(r,z)+\CO(\e^2)
\end{split}
\end{equation*}
with the first coefficient function $H_1(r,z)=\left(H^1_1(r,z),H^2_1(r,z)\right)$ satisfying
\begin{equation*}
\begin{split}
H^1_1(r,z)=&\frac{1}{3 \nu_0^{3/2} \omega ^4}\Big(\omega ^3 \big(2 \pi  \sqrt{3 \mu_1 \nu_0-3 \al_1} \big(\al_1 \nu_1-\al_2\nu_0+\mu_2 \nu_0^2\big)+3 \nu_0^{3/2} r \\
.&(\al_1-\mu_1 \nu_0)\big)-2 \pi  h_1\nu_0 \big(\al_1 \beta_0 \omega  \sqrt{3 \mu_1 \nu_0-3 \al_1}+3 \nu_0^{5/2} z (\al_1-\mu_1 \nu_0)\big)\Big),\\
H^2_1(r,z)=&\frac{2 \pi  \beta_0 h_1z (2 \al_1-\mu_1 \nu_0)}{\omega ^3}.
\end{split}
\end{equation*}
Obviously, a zero of $H\big((r,z),\e\big)$ is a fixed point of the Poincar\'e map $\Pi((r,z),\e)$. Then since the coefficient function $H^1_1(r,z)$ has the simple zero
$$
(\tilde r,\tilde z)=\Bigg(\frac{2 \pi  \left(\omega ^2 (\al_1 \nu_1-\nu_0 (\al_1-\mu_2 \nu_0))-\al_1 \beta_0 h_1 \nu_0\right)}{ \omega ^2 \sqrt{3\nu_0^3( \mu_1 \nu_0- \al_1)}},0 \Bigg),
$$
where
$$
\det(DH_1(\tilde r,\tilde z))=\frac{2 \pi  \beta_0 h_1 \left(2 \al_1^2-3 \al_1 \mu_1 \nu_0+\mu_1^2 \nu_0^2\right)}{\omega ^4},
$$
it follows from the Implicit Function Theorem that $H\big((r,z),\e\big)$ has a branch of zeroes $(\tilde r(\e),\tilde z(\e))$ satisfying $ (\tilde r(0),\tilde z(0))=(\tilde r,\tilde z).$
Consequently, the Poincar\'e map has a fixed point such that
$$
\big(r(\e),z(\e)\big)=\left(2 \sqrt{\frac{\mu_1\nu_0-\al_1}{3 \nu_0}}+\e\tilde r,0\right)+\big(\CO(\e^2),\CO(\e)\big).
$$
The proof of the proposition follows from going back through the changes of variables.
\end{proof}

\section{Numerical Examples}\label{sec:ex}
In this section we present a numerical example for each result stated in Section \ref{sec1}.
\subsection{Example 1} Consider system \eqref{gvd}  with the coefficients \eqref{coefs} satisfying the following relations
\begin{align*}
h(\e)&= 1+\dfrac{\e}{25} + \e^2\dfrac{3}{25}, &
k(\e)&=\frac{67}{50}+\dfrac{\e}{25} + \e^2\dfrac{3}{50}, \nonumber\\
\alpha(\e)&=-\e + \dfrac{\e^2}{25}, &
\beta(\e)&=\frac{117}{50}+\dfrac{\e}{25} + \e^2\dfrac{3}{50},\\
\mu(\e)&=\e\frac{20029}{5025} + \e^2\dfrac{3}{50}, &
\nu(\e)&=1+\dfrac{\e}{25} + \e^2\dfrac{3}{50}.\nonumber
\end{align*}
Take $\e=\frac{1}{70},$  and we observe that $\ell_{1}= -\frac{6403 \pi }{1040}.$ As stated in  Theorem \ref{tvd1}, Figure \ref{toros} shows  the existence of three periodic solutions and two invariant tori coexisting near the origin of coordinates of system \eqref{gvd}.
\begin{figure}
	\begin{overpic}[width=12cm]{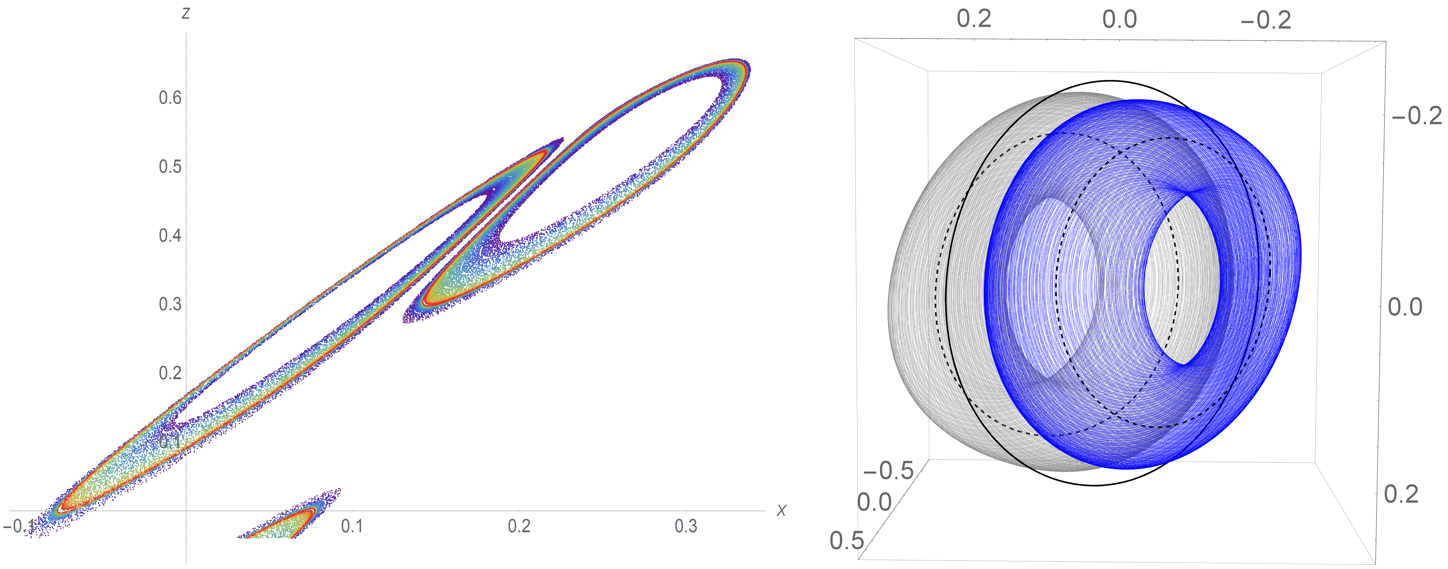}
	\end{overpic}
	\caption{On the left we see the Poincar\'e map of system \eqref{gvd} showing the attracting behaviour near the two invariant tori. On the right, we can see the trajectories starting on $p_\pm=(\pm 0.198, 0, \pm 0.490)$ emulation the two invariant tori. The two symmetric unstable periodic solution are represented by the  dashed curves. The central periodic solution, which has a saddle type of stability, is represented by the continuous curve.}
	\label{toros}	
\end{figure}
\subsection{Example 2} Consider system \eqref{gvd}  with the coefficients \eqref{coefs} satisfying the following relations
\begin{align*}
h(\e)&= -\e + 3\,\e^2, &
k(\e)&=-\frac{2393}{1184}+3\,\e + 3\,\e^2, \nonumber\\
\alpha(\e)&=-\frac{37}{32}-\e + 3\,\e^2, &
\beta(\e)&=-1+3\,\e + 3\,\e^2,\\
\mu(\e)&=-1-2\e + 3\,\e^2, &
\nu(\e)&=\frac{37}{32}-4\e +3\,\e^2.\nonumber
\end{align*}
Take $\e=\dfrac{1}{25}.$  In this case system \eqref{gvd} satisfies the hypothesis {\bf H}$_1$. Thus, as stated in  Theorem \ref{tvd2},  Figure \ref{fam1} shows  the existence of a periodic solutions near the origin of coordinates of system \eqref{gvd}.
\begin{figure}[h!]\centering
	\begin{overpic}[width=6cm]{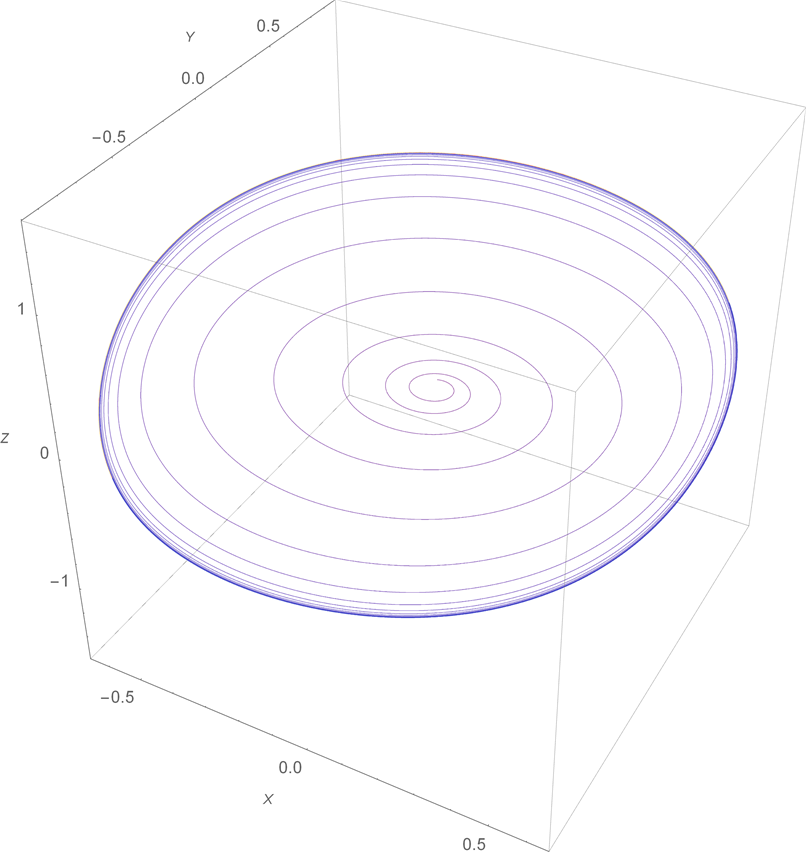}
	\end{overpic}
	\caption{This figure depicts the trajectory starting on $p_1=(0, 0.025, 0.08)$ being attracted by the periodic solution predicted by Theorem \ref{tvd2}.}
	\label{fam1}	
\end{figure}
\subsection{Example 3} Consider system \eqref{gvd}  with the coefficients \eqref{coefs} satisfying the following relations
\begin{align*}
h(\e)&= -1-\e\frac{171}{10} - \e^2\frac{103}{5}, &
k(\e)&=-\frac{2393}{1184}+\e\frac{187}{10}+ 3\e^2, \nonumber\\
\alpha(\e)&=-\frac{37}{32}-\e + \e^2\frac{29}{5}, &
\beta(\e)&=-\e +  \e^2\frac{1}{10},\\
\mu(\e)&=-1-2\,\e - \e^2\frac{221}{10}, &
\nu(\e)&=\frac{37}{32}-4\,\e +\e^2\frac{37}{32}.\nonumber
\end{align*}
Take $\e=\dfrac{1}{25}.$  In this case system \eqref{gvd} satisfies the hypothesis {\bf H}$_2$. Thus, as stated in  Theorem \ref{tvd2},  Figure \ref{fam2} shows  the existence of a periodic solutions near the origin of coordinates of system \eqref{gvd}.
\begin{figure}[h!]\centering
	\begin{overpic}[width=6cm]{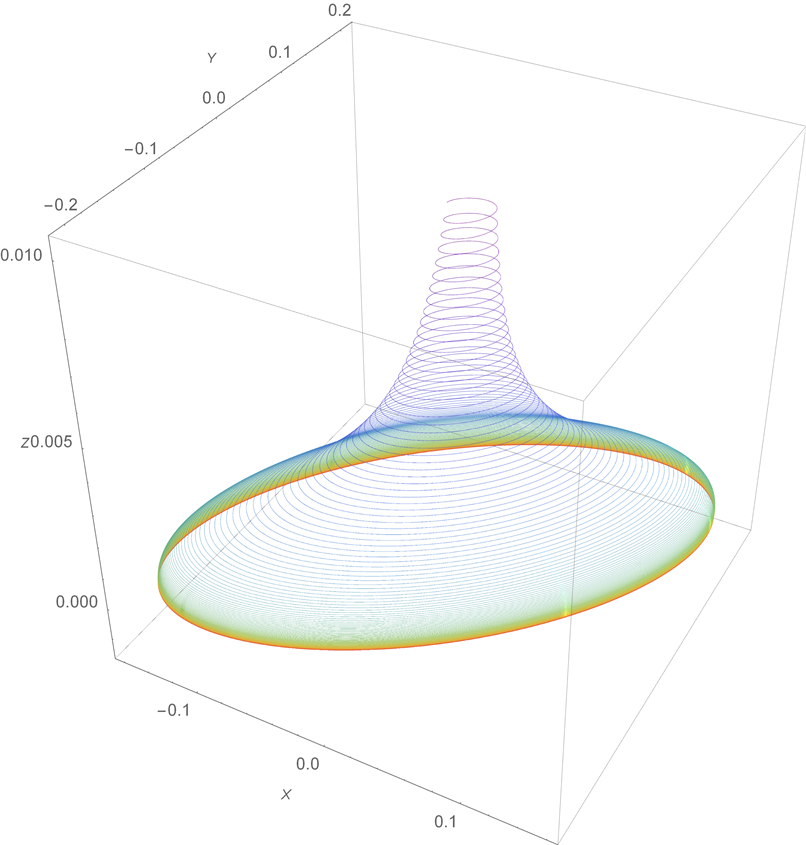}
	\end{overpic}
	\caption{This figure depicts the trajectory starting on $p_2=(0, 0, -0.01)$ being attracted by the periodic solution predicted by Theorem \ref{tvd2}.}
	\label{fam2}	
\end{figure}
\subsection{Example 4} Consider system \eqref{gvd}  with the coefficients \eqref{coefs} satisfying the following relations
\begin{align*}
h(\e)&=1 + \e^2\frac{1}{50}, &
k(\e)&=-1-128\,\e+ \e^2\frac{1}{50}, \nonumber\\
\alpha(\e)&=-25\,\e^3, &
\beta(\e)&=1 + \e^2\frac{1}{50},\\
\mu(\e)&=\e+\e^2\frac{4065}{64}, &
\nu(\e)&=128\,\e -\e^2.\nonumber
\end{align*}
Take $\e=\dfrac{1}{150}.$  In this case system \eqref{gvd} satisfies the hypothesis {\bf H}$_3$. Thus, as stated in  Theorem \ref{tvd2},  Figure \ref{fam4} shows  the existence of a periodic solution near the origin of coordinates of system \eqref{gvd}.
\begin{figure}[h!]\centering
	\begin{overpic}[width=6cm]{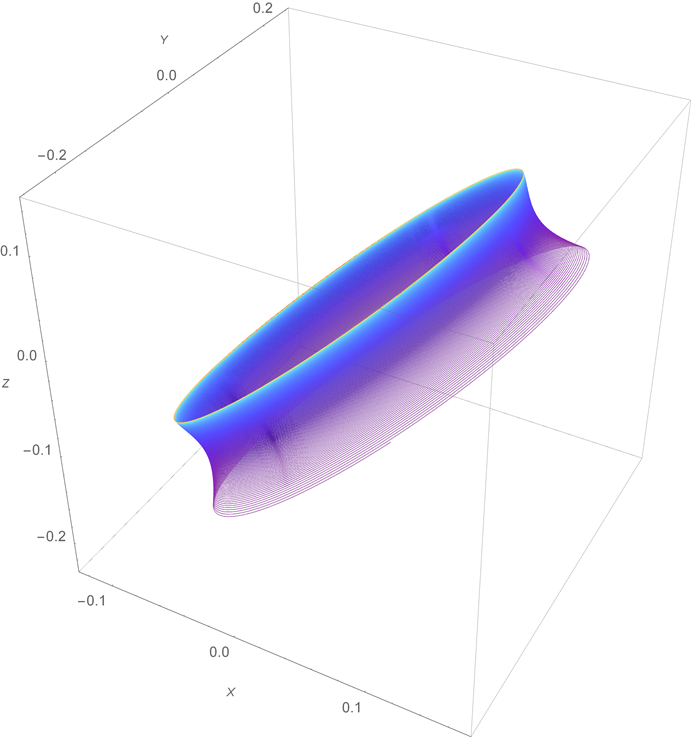}
	\end{overpic}
	\caption{This figure depicts the trajectory starting on $p_3=(0.1, -0.2, 0)$ being attracted by the periodic solution predicted by Theorem \ref{tvd2}.}
	\label{fam4}	
\end{figure}

\subsection{Example 5} Consider system \eqref{gvd}  with the coefficients \eqref{coefs} satisfying the following relations
\begin{align*}
h(\e)&= -\e + \e^2\frac{43}{5}, &
k(\e)&=-\frac{16}{5}-\e\frac{54}{5}+ \e^2\frac{8}{5}, \nonumber\\
\alpha(\e)&=-\e, &
\beta(\e)&=-1-\e\frac{14}{5} +\e^2\frac{1}{5},\\
\mu(\e)&=-\e , &
\nu(\e)&=\frac{5}{16}-\e -\e^2\frac{46}{5}.\nonumber
\end{align*}
Here, we take $\e=\dfrac{1}{150}.$  In this case system \eqref{gvd} satisfies the hypothesis {\bf H}$_4$. Thus, as stated in  Theorem \ref{tvd2},  Figure \ref{fam5} shows  the existence of a periodic solution near the origin of coordinates of system \eqref{gvd}.
\begin{figure}[h!]\centering
	\begin{overpic}[width=6cm]{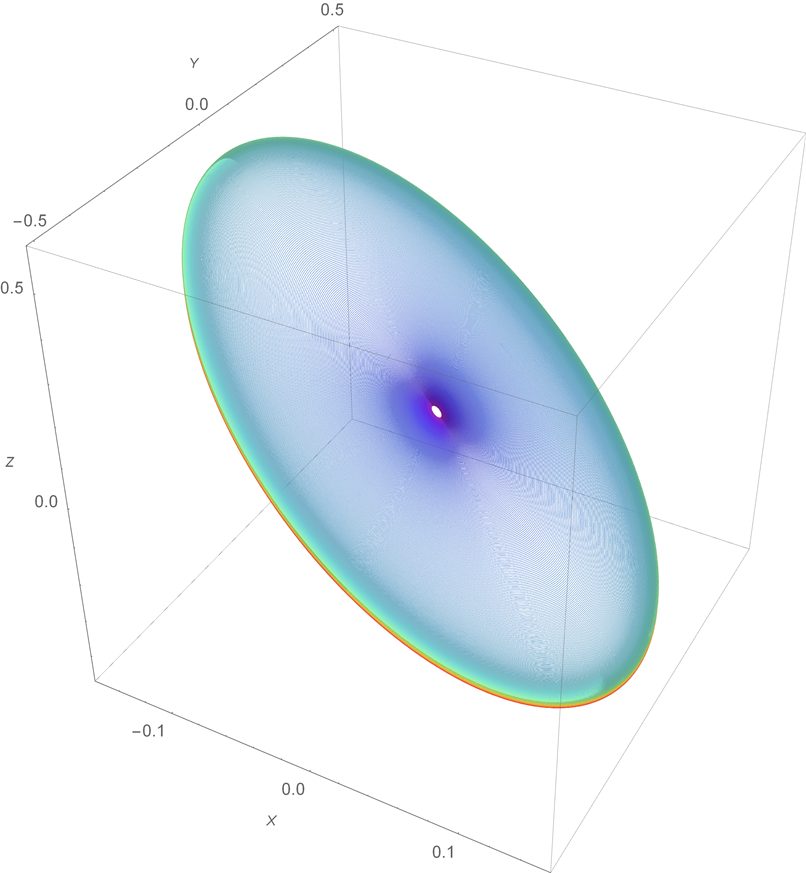}
	\end{overpic}
	\caption{This figure depicts the trajectory starting on $p_4=(0,0.1, 0.08)$ being attracted by the periodic solution predicted by Theorem \ref{tvd2}.}
	\label{fam5}	
\end{figure}

\section*{Acknowledgements}
This work started when MRC was doing a research visit to the  Instituto Superior T\'{e}cnico - Universidade de Lisboa, supported by the grant Becas Iberoam\'{e}rica, Santander Investigaci\'{o}n. MRC is gratefully indebted to Prof. C. Valls for her supervision during this period the for the fruitful collaboration. MRC is partially supported by FAPESP grants 2018/07344-0 and 2019/05657-4. CV is partially supported by FCT/Portugal through UID/MAT/ 04459/2019.


\end{document}